\documentclass[oneside,english]{amsart}
\usepackage[T1]{fontenc}
\usepackage[latin9]{inputenc}
\usepackage{textcomp}
\usepackage{amsthm}
\usepackage{amstext}
\usepackage{amssymb}
\usepackage{graphicx}
\usepackage{esint}
\usepackage[numbers]{natbib}

\makeatletter
\numberwithin{equation}{section}
\numberwithin{figure}{section}
\theoremstyle{plain}
\newtheorem{thm}{\protect\theoremname}
  \theoremstyle{definition}
  \newtheorem{defn}[thm]{\protect\definitionname}
  \theoremstyle{definition}
  \newtheorem{example}[thm]{\protect\examplename}
  \theoremstyle{plain}
  \newtheorem{prop}[thm]{\protect\propositionname}
  \theoremstyle{remark}
  \newtheorem{rem}[thm]{\protect\remarkname}
  \theoremstyle{plain}
  \newtheorem{lem}[thm]{\protect\lemmaname}
\newenvironment{lyxlist}[1]
{\begin{list}{}
{\settowidth{\labelwidth}{#1}
 \setlength{\leftmargin}{\labelwidth}
 \addtolength{\leftmargin}{\labelsep}
 }}
{\end{list}}

\makeatother

\usepackage{babel}
  \providecommand{\definitionname}{Definition}
  \providecommand{\examplename}{Example}
  \providecommand{\lemmaname}{Lemma}
  \providecommand{\propositionname}{Proposition}
  \providecommand{\remarkname}{Remark}
\providecommand{\theoremname}{Theorem}

\begin{document}

\title[Stability criteria for partitioning problems]{Stability Criteria for Multiphase Partitioning Problems with Volume
Constraints}

\author{N. D. Alikakos({*})}

\address{Department of Mathematics, University of Athens, Panepistemiopolis,
15784 Athens, Greece}

\email{nalikako@math.uoa.gr}

\author{A. C. Faliagas}

\address{Department of Mathematics, University of Athens, Panepistemiopolis,
15784 Athens, Greece}

\email{afaliaga@math.uoa.gr}

\dedicatory{Dedicated to the memory of Paul Fife.}

\thanks{({*}) The first author was partially supported through the project
PDEGE (Partial Differential Equations Motivated by Geometric Evolution),
co-financed by the European Union European Social Fund (ESF) and national
resources, in the framework of the program Aristeia of the Operational
Program Education and Lifelong Learning of the National Strategic
Reference Framework (NSRF)}
\begin{abstract}
We study the stability of partitions involving two or more phases
in convex domains under the assumption of at most two-phase contact,
thus excluding in particular triple junctions. We present a detailed
derivation of the second variation formula with particular attention
to the boundary terms, and then study the sign of the principal eigenvalue
of the Jacobi operator. We thus derive certain stability criteria,
and in particular we recapture the Sternberg-Zumbrun result on the
instability of the disconnected phases in the more general setting
of several phases.
\end{abstract}
\maketitle

\section{Introduction}

The partitioning  of a set into a number of subsets (the ``phases'')
so that the dividing hypersurface (the ``interface'') has minimal
area, is a problem of geometric analysis and calculus of variations.
It is of high importance in the physical sciences and engineering
because of its relation to surface tension. Examples include a variety
of phenomena\citep{key-1} ranging from the annealing of metals (Mullins\citep{key-2})
to the segregation of biological species (Ei et al\citep{key-3}).
Two phase systems formed by the mixing of two polymers or a polymer
and a salt in water are used for the separation of cells, membranes,
viruses, proteins, nucleic acids, and other biomolecules. The partitioning
between the two phases is dependent on the surface properties of the
materials. An overview of the physical aspects of the subject is offered
in \citep{key-4}. Early studies of the mathematical problem of partitioning
include Nitsche's paper\citep{key-5,key-6}, and Almgren's Memoir\citep{key-7}
(see also White\citep{key-23}.)

Paul Fife was one of the top applied mathematicians of his time with
significant and lasting contributions to Diffuse Waves, Diffuse Interfaces,
Stefan problems and Phase Field Models. His monographs\citep{key-111,key-112}
``Dynamics of Internal Layer and Diffuse Interfaces'' and ``Mathematical
Aspects of Reactions and Diffusive Systems'' and the IMA volume\citep{key-113}
``Dynamical Issues in Combustion Theory'' are classics. Fife with
his collaborators studied extensively the dynamical problems related
to the generation of partition and to their coarsening. For a sample
see \citep{key-8,key-9,key-101,key-102}.

Sternberg-Zumbrun\citep{key-15}, treated the static problem and proved
that disconnected two phase partitions of convex sets are always unstable.
The S-Z formulas with little notation changes are given in the next
theorem. For the reader's convenience some details are given to make
the exposition self-contained. Throughout this paper, we take $\Omega\subset\mathbb{R}^{N}$
to be a bounded domain with smooth boundary $\Sigma=\partial\Omega$. 
\begin{defn}
Let $M$ be a $n$-dimensional $C^{1}$ submanifold of $\mathbb{R}^{N}$
with boundary and $V$ an open subset of $\mathbb{R}^{N}$ such that
$V\cap M\neq\emptyset$. A variation of $M$ is a collection of diffeomorphisms
$(\xi^{t})_{t\in I}$, $I=]-\delta,\delta[$, $\delta>0$, $\xi^{t}:V\to V$
such that

(i) The function $\xi(x,t)=\xi^{t}(x)$ is $C^{2}$ 

(ii) $\xi^{0}=id_{V}$ 

(iii) $\xi^{t}|_{V\setminus K}=id_{V\setminus K}$ for some compact
set $K\subset V$ .
\end{defn}
In place of the $\xi^{t}$ we often consider their extensions by identity
to all of $\mathbb{R}^{N}$. With each variation we associate the
\emph{first} and \emph{second variation fields}
\[
w(x)=\xi_{t}(x,0),\quad a(x)=\xi_{tt}(x,0)
\]
known\citep{key-16} as \emph{velocity} and \emph{acceleration fields},
$\xi_{t}$, $\xi_{tt}$ denoting first and second partial derivative
in $t$.

As $M\subset\Omega$, $\partial M\subset\Sigma$, admissible variations
of $M$ should respect the rigidity of the boundary of $\Omega$.
In this connection S-Z suggested that admissible variations of $M$
be obtained by solving the ODE
\begin{equation}
\frac{d\xi}{dt}=w(\xi),\quad\xi(0)=x\label{eq:SZ-var}
\end{equation}
for any given first variation vector field $w$ and then setting $\xi^{t}(x)=\xi_{x}(t)$,
where $\xi_{x}$ is the solution of (\ref{eq:SZ-var}) for the initial
condition $\xi_{x}(0)=x$. The requirement for rigid container walls
is satisfied by selecting $w$ so that $w(p)\in T_{p}\Sigma$ for
all $p\in\Sigma$, $T_{p}X$ denoting as usual the tangent space of
$X$. In this paper we consider only \emph{normal variations}, i.e.
those satisfying $w(p)\in N_{p}(M)$ for all $p\in M$. $N_{p}$ is
the normal bundle of $M$ at $p$.

By a \emph{partitioning} of $\Omega$ we mean a collection $M=(M_{i})_{i=1}^{m}$
of $C^{2}$ hypersurfaces with boundary (which is again $C^{2}$),
which are non-intersecting and their boundaries lie in $\Sigma=\partial\Omega$.
Additionally, by a \emph{minimal partitioning} $M$ we mean a critical
point of the area functional $A$ under the assumed volume constraints,
i.e. 
\[
\delta A(M):=\left.\frac{d}{dt}A(M^{t})\right|_{t=0}=0
\]
for all variations preserving $\Sigma$ and the volume of the phases.
In this equation $M^{t}=\xi^{t}(M)$, $\xi^{t}$ being a variation,
and $A(M)$ is the area of $M$.
\begin{thm}[Sternberg-Zumbrun]
\label{thm:SZ}Let $M=(M_{i})_{i=1}^{r}$ be a minimal 2-phase partitioning
in $\Omega$. Then for any normal variation of $M$, which preserves
$\Sigma$ and the volume of the phases, i.e.
\[
\int_{M}f=0,
\]
the second variation of area of $M$ is given by
\[
\delta^{2}A(M)=\left.\frac{d^{2}}{dt^{2}}A(M^{t})\right|_{t=0}=\int_{M}\left(|\mathrm{grad}_{M}\, f\,|^{2}-|B_{M}|^{2}f^{2}\right)-\int_{M\cap\Sigma}\sigma f^{2}
\]
where $f$ is the projection of the first variation field $w$ on
the unit normal field $N$ of $M$, $|B_{M}|^{2}$ is the norm of
the second fundamental form $B_{M}$ associated with $M$ and $\sigma=II_{\Sigma}(N,N)$
is the scalar version of the second fundamental form of $\Sigma$.
\end{thm}
The orientation of $M$ is selected so that $\sigma\geqslant0$ on
$\Sigma$ for convex $\Omega$.

Using this formula with $r=2$, S-Z proved that for convex $\Omega$,
when $\sigma\neq0$ on $\Sigma$, every disconnected two phase partitioning
is necessarily unstable. Recall that by definition a minimal partitioning
is \emph{stable} when $\delta^{2}A(M)>0$ for all variations $w\neq0$
preserving $\Sigma$ and the volume of the phases.

In the following section an extension of the S-Z formula to $m$-phase
problems is given in Proposition \ref{prop:gen-2nd-var}. The instability
of disconnected multiphase partitions follows as an application of
this. In Section \ref{sec:Spectr-Anal} we develop the spectral theory
of the bilinear form expressing the second variation of area, which
is the main tool for proving our stability/instability results. The
main statement in this section is Proposition \ref{prop:Categ-b}
which states that, for normalized variations, the minimum of the second
variation of area is given by the principal eigenvalue. The difficulties
in obtaining this result are (i) that the the boundary integral $\int_{\partial M}f^{2}$
cannot be bounded above by $\int_{M}f^{2}$, and (ii) that admissible
variations need not satisfy the boundary condition (\ref{eq:eigv-BC})
of the related eigenvalue problem. They were handled by developing
an interpolation estimate for the boundary integral in Lemma \ref{lem:interp-est}.
An extension of Proposition \ref{prop:Categ-b} to $m$-phase partitioning
problems is immediate. Proposition \ref{prop:reduc} gives a characterization
of all connected $m$-phase partitionings by reduction to the two
phase case.

The last two sections are devoted to applications. In Section \ref{sec:Apps-RN}
we prove the existence of unstable partitionings in $\mathbb{R}^{N}$
when $\Omega$ satisfies hypothesis (H) (see Section \ref{sec:Apps-RN})
and that spherical partitionings are stable when they do not make
contact with the boundary of $\Omega$. As a byproduct we also prove
that spherical partitionings in bounded sets are never absolute minimizers
of the area functional under volume constraint. Applications to 2-dimensional
problems have been included in Section \ref{sec:Apps-2D}. The main
results here are criteria for instability, Proposition \ref{prop:crit-1},
and stability, Proposition \ref{prop:crit-2}. Proposition \ref{prop:no-nec-and-suf-cond}
shows that sufficiently small partitions are stable. For related work
see \citep{key-11,key-12,key-13}.

\section{Multiphase Partitioning Problems}

A more general functional may be used for more than two phases:
\begin{equation}
A(M)=\sum_{i=1}^{r}\gamma_{i}A(M_{i}).\label{eq:3-phase-area}
\end{equation}
The coefficients $\gamma_{i}>0$ have the physical meaning of \emph{surface
energy density}. The summation in (\ref{eq:3-phase-area}) extends
over all interfaces constituting the partition problem. The collection
$M=(M_{i})_{i=1}^{r}$ will be considered oriented, its orientation
being determined by the orientations of the $M_{i}$. There are $2^{r}$
possible orientations for $M$. Most of the following formulas depend
on the orientation of $M$. Admissible variations for the functional
in (\ref{eq:3-phase-area}) are those preserving phase volume. They
can be directly obtained from general variations by rendering them
volume preserving (see \citep{key-15}). As this is a highly involved
process for multiphase systems, we use the method of Lagrange multipliers,
which, as it turns out, is more convenient. In this connection we
consider the following modified (weighted) functional:
\begin{equation}
A^{\star}(M)=A(M)-\sum_{j=1}^{m}\lambda_{j}\left(\sum_{k=1}^{P_{j}}\left|\Omega_{jk}\right|-V_{j}\right)\label{eq:modif-area-func}
\end{equation}
In this formula $\left|\cdot\right|$ denotes volume, $m$ is the
number of phases, $P_{j}$ is the number of distinct sets phase $j$
is split (indexed by $k$), $V_{j}$ is the volume of phase $j$ and
$\lambda_{j}$ is the Lagrange multiplier corresponding to the volume
constraint for the $j$-th phase. Since $\sum_{j=1}^{m}\sum_{k=1}^{P_{j}}\left|\Omega_{jk}\right|=\left|\Omega\right|$,
there are only $m-1$ linearly independent constraints and we could
have used only $m-1$ Lagrange multiplies. For convenience and as
the final results are identical, we use $m$ multipliers.
\begin{example}
In the case of a disconnected 3-phase partition the weighted area
functional is given by (see Fig. \ref{fig:Disc-3-ph-part})
\begin{figure}
\includegraphics[bb=0bp -1bp 574bp 348bp,scale=0.5]{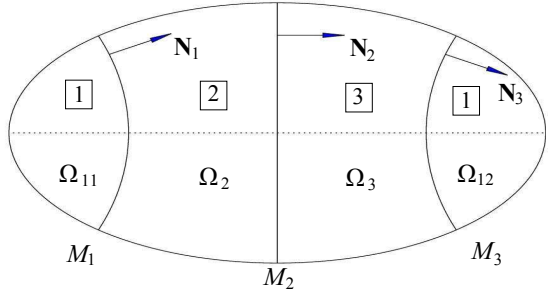}

\caption{\label{fig:Disc-3-ph-part}Disconnected 3-phase partitioning. Numbers
in boxes indicate phases. $M_{i}$ are the interfaces and $\mathbf{N}_{i}$
are corresponding normal fields.}
\end{figure}
\begin{equation}
A^{\star}(M)=\sum_{i=1}^{3}\gamma_{i}A(M_{i})-\lambda_{1}\left(\left|\Omega_{11}\right|+\left|\Omega_{12}\right|\right)-\lambda_{2}\left|\Omega_{2}\right|-\lambda_{3}\left|\Omega_{3}\right|\label{eq:example-func}
\end{equation}
The volume constants $V_{j}$ were dropped as they play no part in
the variational process.

The following proposition extends Theorem \ref{thm:SZ} to more than
two phases. \end{example}
\begin{prop}
\label{prop:gen-2nd-var}Let $M=(M_{i})_{i=1}^{r}$ be a minimal multiphase
partitioning with volume constancy constraint for the phases. Further
let $N_{i}$ be the unit normal field of $M_{i}$. Then

(i) The scalar mean curvature $\kappa_{i}=H_{i}\cdot N_{i}$ of each
interface $M_{i}$ is constant.

(ii) The scalar mean curvatures satisfy the relation
\begin{equation}
\sum_{j=1}^{r}\gamma_{j}\kappa_{j}=0\label{eq:mean-curv-id}
\end{equation}

(iii) Each $M_{i}$ is normal to $\Sigma$, i.e. on each $M_{i}\cap\Sigma$
we have $N_{i}\cdot N_{\Sigma}=0$ or $N_{i}(p)\in T_{p}\Sigma$ for
all $p\in\partial M_{i}$. \textup{$N_{\Sigma}$} is the normal field
of $\Sigma$.

(iv) For any admissible variation of $M$, i.e. one preserving $\Sigma$
and the volume of the phases, the second variation of area of $M$
is given by
\[
\delta^{2}A(M)=\sum_{i=1}^{r}\gamma_{i}\int_{M_{i}}\left(|\mathrm{grad}_{M_{i}}\, f_{i}\,|^{2}-|B_{M_{i}}|^{2}f_{i}^{2}\right)-\sum_{i=1}^{r}\gamma_{i}\int_{\partial M_{i}}\sigma_{i}\, f_{i}^{2}
\]
where $\sigma_{i}=II_{\Sigma}(N_{i},N_{i})$.\end{prop}
\begin{proof}
For concreteness we consider the disconnected 3-phase partitioning
of Fig. \ref{fig:Disc-3-ph-part} with the indicated orientation.
Let $w$ be any variation of $M$. By (\ref{eq:example-func}), the
formula for the first variation of the area of a manifold \citep{key-16}
$\delta A(M_{i})=\int_{M_{i}}\mathrm{div}_{M_{i}}w$ and 
\[
\delta\left|\Omega_{jk}\right|=\int_{\partial\Omega_{jk}}w\cdot N_{\partial\Omega_{jk}},
\]
$N_{\partial\Omega_{jk}}$ being the unit \emph{outward} normal field
of $\partial\Omega_{jk}$, which is easily established, we obtain:
\[
\begin{alignedat}{1}\delta A^{\star}(M)= & \sum_{i=1}^{3}\gamma_{i}\int_{M_{i}}\mathrm{div}_{M_{i}}w-\lambda_{1}\left(\int_{M_{1}}w\cdot N_{1}-\int_{M_{3}}w\cdot N_{3}\right)\\
 & -\lambda_{2}\left(\int_{M_{2}}w\cdot N_{2}-\int_{M_{1}}w\cdot N_{1}\right)\\
 & -\lambda_{3}\left(\int_{M_{3}}w\cdot N_{3}-\int_{M_{2}}w\cdot N_{2}\right).
\end{alignedat}
\]
Let $H_{i}$ be the mean curvature vector field of $M_{i}$, $\nu_{i}$
the unit tangent field of $M_{i}$ which is normal to $\partial M_{i}$
(also known as ``conormal'' field) and $f_{i}=w\cdot N_{i}$ the
normal component of the variation field on $M_{i}$ (tangential components
are irrelevant and are disregarded from the outset). Let $(\cdot)^{\top}$
denote projection on the tangent space of a manifold. Application
of the identity \citep{key-16}
\begin{equation}
\mathrm{div}_{M}w=\mathrm{div}_{M}w^{\top}-H\cdot w^{\bot},\label{eq:div-formula}
\end{equation}
the divergence theorem for manifolds \citep{key-16}, $\int_{M_{i}}\mathrm{div}_{M}w^{\top}=\int_{\partial M_{i}}w\cdot\nu_{i}$,
and reordering of terms give:
\[
\begin{alignedat}{1}\delta A^{\star}(M)= & \sum_{i=1}^{3}\gamma_{i}\int_{\partial M_{i}}w\cdot\nu_{i}-\int_{M_{1}}(\gamma_{1}\kappa_{1}+\lambda_{1}-\lambda_{2})f_{1}\\
 & -\int_{M_{2}}(\gamma_{2}\kappa_{2}+\lambda_{2}-\lambda_{3})f_{2}\\
 & -\int_{M_{3}}(\gamma_{3}\kappa_{3}+\lambda_{3}-\lambda_{1})f_{3}.
\end{alignedat}
\]
Standard arguments render 
\begin{equation}
\gamma_{1}\kappa_{1}=\lambda_{2}-\lambda_{1},\quad\gamma_{2}\kappa_{2}=\lambda_{3}-\lambda_{2},\quad\gamma_{3}\kappa_{3}=\lambda_{1}-\lambda_{3}\label{eq:Lagr-mult}
\end{equation}
which prove (i) and (ii) and $w\cdot\nu_{i}=0$, which proves (iii).
By variation of the Lagrange multipliers we recover the volume constancy
constraints:
\[
\int_{M_{1}}f_{1}=\int_{M_{2}}f_{2}=\int_{M_{3}}f_{3}
\]

For the proof of (iv) we start by Simon's general formula for the
second variation of area \citep{key-16} 
\begin{equation}
\begin{alignedat}{1}\delta^{2}A(M_{i}) & =\int_{M_{i}}\left[\mathrm{div}_{M_{i}}a+(\mathrm{div}_{M_{i}}w)^{2}+g^{rs}\,\left\langle (D_{E_{r}}w)^{\bot},(D_{E_{s}}w)^{\bot}\right\rangle \right.\\
 & \left.-g^{rk}g^{sl}\left\langle D_{E_{r}}w,E_{s}\right\rangle \left\langle D_{E_{l}}w,E_{k}\right\rangle \right]
\end{alignedat}
\label{eq:SIM-2nd-var}
\end{equation}
In this formula $w$, $a$ are the first and second variation fields,
$(\cdot)^{\top}$ denotes projection on the tangent space of $M_{i}$,
$(\cdot)^{\bot}$ denotes projection on the normal space of $M_{i}$,
$E_{1},\cdots,E_{N-1}$ are the basis vector fields in a chart, $g_{rs}=E_{r}\cdot E_{s}$
are the covariant components of the metric tensor of $M_{i}$ and
$g^{kl}$ its contravariant components. Summation convention applies
throughout this paper. The notation $\left\langle \cdot,\cdot\right\rangle $
is alternatively used to denote scalar product in lengthier expressions.
Recalling that $w$ is normal to $M_{i}$, i.e. $w^{\bot}=w=f_{i}N_{i}$,
we have
\[
g^{rs}\left\langle (D_{E_{r}}w)^{\bot},(D_{E_{s}}w)^{\bot}\right\rangle =g^{rs}(D_{E_{r}}f_{i})(D_{E_{s}}f_{i})=|\mathrm{grad}_{M_{i}}\, f_{i}\,|^{2}
\]
and
\begin{alignat*}{1}
g^{rk}g^{sl}\left\langle D_{E_{r}}w,E_{s}\right\rangle \left\langle D_{E_{l}}w,E_{k}\right\rangle  & =g^{rk}g^{sl}\left\langle f_{i}D_{E_{r}}N_{i},E_{s}\right\rangle \left\langle f_{i}D_{E_{l}}N_{i},E_{k}\right\rangle \\
 & =f_{i}^{2}g^{ik}g^{jl}\left\langle D_{E_{r}}N_{i},E_{s}\right\rangle \left\langle D_{E_{l}}N_{i},E_{k}\right\rangle \\
 & =f_{i}^{2}g^{ik}g^{jl}\left\langle N_{i},D_{E_{r}}E_{s}\right\rangle \left\langle N_{i},D_{E_{l}}E_{k})\right\rangle \\
 & =f_{i}^{2}B_{\, r}^{k}B_{\, k}^{r}=f_{i}^{2}|B_{M_{i}}|^{2}
\end{alignat*}
In this equation $B_{rk}=\left\langle N,D_{E_{r}}E_{k}\right\rangle $
are the covariant components of the second fundamental form tensor
$II(u,v)=\left\langle N,D_{u}v\right\rangle $ ($u$, $v$ are tangent
vector fields) and $B_{\, r}^{s}=g^{sk}B_{rk}$. By (\ref{eq:div-formula})
as $w^{\top}=0$ we obtain
\begin{equation}
\mathrm{div}_{M_{i}}w=-\kappa_{i}\, f_{i}.\label{eq:divw}
\end{equation}
Combination of (\ref{eq:div-formula}) with the equality $a=D_{w}w$,
which is obtained by taking the time derivative of (\ref{eq:SZ-var}),
gives for $\mathrm{div}_{M_{i}}a$:
\[
\mathrm{div}_{M_{i}}a=\mathrm{div}_{M_{i}}a^{\top}-\kappa_{i}N_{i}\cdot D_{w}w
\]
By Lemma \ref{lem:var-vol} we have for the second variation of the
$|\Omega_{jk}|$:
\[
\delta^{2}|\Omega_{11}|=\int_{M_{1}}f_{1}\mathrm{div}_{\mathbb{R}^{N}}w,\quad\delta^{2}|\Omega_{12}|=-\int_{M_{3}}f_{3}\mathrm{div}_{\mathbb{R}^{N}}w
\]
\[
\delta^{2}|\Omega_{2}|=\int_{M_{2}}f_{2}\mathrm{div}_{\mathbb{R}^{N}}w-\int_{M_{1}}f_{1}\mathrm{div}_{\mathbb{R}^{N}}w
\]
\[
\delta^{2}|\Omega_{3}|=\int_{M_{3}}f_{3}\mathrm{div}_{\mathbb{R}^{N}}w-\int_{M_{2}}f_{2}\mathrm{div}_{\mathbb{R}^{N}}w
\]
By (\ref{eq:example-func})
\[
\delta^{2}A^{\star}(M)=\sum_{i=1}^{3}\gamma_{i}\delta^{2}A(M_{i})-\lambda_{1}(\delta^{2}|\Omega_{11}|+\delta^{2}|\Omega_{12}|)-\lambda_{2}\delta^{2}|\Omega_{2}|-\lambda_{3}\delta^{2}|\Omega_{3}|
\]
Replacing for $\delta^{2}A(M_{i})$, $\delta^{2}|\Omega_{jk}|$ by
the above equalities and rearranging give the following expression
for $\delta^{2}A^{\star}(M)$:
\[
\begin{alignedat}{1}\delta^{2}A^{\star}(M)= & \sum_{i=1}^{3}\gamma_{i}\int_{M_{i}}(|\mathrm{grad}_{M_{i}}\, f_{i}\,|^{2}-|B_{M_{i}}|^{2}f_{i}^{2})\\
 & +\int_{M_{1}}\left[\gamma_{1}\mathrm{div}_{M_{1}}a^{\top}-\gamma_{1}\kappa_{1}N_{1}\cdot D_{w}w+\gamma_{1}\kappa_{1}^{2}f_{1}^{2}-(\lambda_{1}-\lambda_{2})f_{1}\mathrm{div}_{\mathbb{R}^{N}}w\right]\\
 & +\int_{M_{2}}\left[\gamma_{2}\mathrm{div}_{M_{2}}a^{\top}-\gamma_{2}\kappa_{2}N_{2}\cdot D_{w}w+\gamma_{2}\kappa_{2}^{2}f_{2}^{2}-(\lambda_{2}-\lambda_{3})f_{2}\mathrm{div}_{\mathbb{R}^{N}}w\right]\\
 & +\int_{M_{3}}\left[\gamma_{3}\mathrm{div}_{M_{3}}a^{\top}-\gamma_{3}\kappa_{3}N_{3}\cdot D_{w}w+\gamma_{3}\kappa_{3}^{2}f_{3}^{2}-(\lambda_{3}-\lambda_{1})f_{3}\mathrm{div}_{\mathbb{R}^{N}}w\right]
\end{alignedat}
\]
By (\ref{eq:Lagr-mult}) the integral on the second row assumes the
form
\[
\begin{gathered}\int_{M_{1}}\left[\gamma_{1}\mathrm{div}_{M_{1}}a^{\top}-\gamma_{1}\kappa_{1}N_{1}\cdot D_{w}w+\gamma_{1}\kappa_{1}^{2}f_{1}^{2}+\gamma_{1}\kappa_{1}f_{1}\mathrm{div}_{\mathbb{R}^{N}}w\right]=\\
\int_{M_{1}}\left[\gamma_{1}\mathrm{div}_{M_{1}}a^{\top}-\gamma_{1}\kappa_{1}f_{1}(N_{1}\cdot D_{N_{1}}w-\mathrm{div}_{\mathbb{R}^{N}}w)+\gamma_{1}\kappa_{1}^{2}f_{1}^{2}\right]=\\
\int_{M_{1}}\left[\gamma_{1}\mathrm{div}_{M_{1}}a^{\top}+\gamma_{1}\kappa_{1}f_{1}\mathrm{div}_{M_{1}}w+\gamma_{1}\kappa_{1}^{2}f_{1}^{2}\right]=\\
\int_{M_{1}}\left[\gamma_{1}\mathrm{div}_{M_{1}}a^{\top}-\gamma_{1}\kappa_{1}^{2}f_{1}^{2}+\gamma_{1}\kappa_{1}^{2}f_{1}^{2}\right]=\int_{M_{1}}\gamma_{1}\mathrm{div}_{M_{1}}a^{\top}
\end{gathered}
\]
On the second and third equalities we have used the identity \citep{key-16}
\[
\mathrm{div}_{\mathbb{R}^{n}}X=\mathrm{div}_{M}X+N\cdot D_{N}X,
\]
and Equation (\ref{eq:divw}); $X$ is any differentiable field on
$M$. Application of the divergence theorem gives
\[
\begin{aligned}\int_{M_{1}}\gamma_{1}\mathrm{div}_{M_{1}}a^{\top}=\gamma_{1}\int_{\partial M_{1}}a\cdot\nu=\gamma_{1}\int_{\partial M_{1}}\nu\cdot D_{w}w=\\
\gamma_{1}\int_{\partial M_{1}}f_{1}^{2}\nu\cdot D_{N_{1}}N_{1}=-\gamma_{1}\int_{\partial M_{1}}f_{1}^{2}II_{\Sigma}(N_{1},N_{1})
\end{aligned}
\]
In a similar fashion we reformulate the last two rows in the expression
for $\delta^{2}A^{\star}(M)$ and this completes the proof of (iv).\end{proof}
\begin{rem}
The convention for the second fundamental form of $\Sigma=\partial\Omega$
is $N_{\Sigma}=-\nu$ so that $II_{\Sigma}(N_{i},N_{i})$ is always
non-negative for convex $\Omega$. $N_{i}$ is the normal field of
$M_{i}$, which is tangent to $\Sigma$.
\end{rem}

\begin{rem}
The equation of (ii) of the proposition depends on orientation choice.
In the case of a disconnected 2-phase partition we have $\gamma_{1}=\gamma_{2}=\gamma_{12}$,
the interfacial energy density of phases 1 and 2, and (ii) reduces
to $\kappa_{1}+\kappa_{2}=0$, which, in the 2-dimensional case, implies
the interfaces are circular arcs of \emph{equal} radii. This restricts
considerably the number of possible realizations of minimal disconnected
multiphase partitions.\end{rem}
\begin{example}
Let $\Omega$ be an ellipse centered at 0 with major and minor semiaxes
$a$, $b$. For $a>x_{0}>0$ the tangents at $(x_{0},\pm y_{0})$
have equations
\[
\frac{x_{0}}{a^{2}}(x-x_{0})+\frac{y_{0}}{b^{2}}(y-y_{0})=0,\quad\frac{x_{0}}{a^{2}}(x-x_{0})-\frac{y_{0}}{b^{2}}(y+y_{0})=0.
\]
These lines intersect at $x=x_{0}+\frac{a^{2}y_{0}^{2}}{b^{2}x_{0}}=\frac{a^{2}}{x_{0}}$.
Hence the radius of the circular arc, which intersects the ellipse
at right angles, is given by
\[
R^{2}=\left(\frac{a^{2}}{x_{0}}-x_{0}\right)^{2}+y_{0}^{2}=\left(\frac{a^{4}}{b^{4}}\frac{y_{0}^{2}}{x_{0}^{2}}+1\right)y_{0}^{2}
\]
This is a monotone decreasing function of $x_{0}$ and from this it
follows, with simple geometric arguments, that all possible minimal,
disconnected, 2-phase partitionings of an ellipse, are pairs of transversal
circular arcs symmetric about the $x$ or $y$-axis. \hfill{}$\square$
\end{example}
As an application of Proposition \ref{prop:gen-2nd-var} we prove
the instability of disconnected 3-phase partitions in a convex set.
Given any such partition, we choose a variation which is constant
on each interface. The volume constraints are satisfied if we choose
\[
f_{i}=\frac{1}{A(M_{i})},\; i=1,2,3.
\]
By (iv) of Proposition \ref{prop:gen-2nd-var} we obtain
\[
\delta^{2}A(M)=-\sum_{i=1}^{3}\frac{\gamma_{i}}{A(M_{i})^{2}}\int_{M_{i}}|B_{M_{i}}|^{2}-\sum_{i=1}^{3}\frac{\gamma_{i}}{A(M_{i})^{2}}\int_{\partial M_{i}}\sigma_{i}
\]
which is negative when $\sigma_{i}\neq0$. Generalization to an arbitrary
number of phases and phase splittings is immediate: 
\begin{prop}
\label{prop:SZ-gen}Let $\Omega\subset\mathbb{R}^{N}$ be an open,
bounded and convex set with $C^{2,\alpha}$ boundary. Let $M=(M_{i})_{i=1}^{r}$
be any stable $m$-phase partitioning of $\Omega$. Further, assume
that $\Omega$ is strictly convex, in particular $II_{\partial\Omega}(N_{i},N_{i})>0$
at all points of $\partial M_{i}\cap\Sigma$, $i=1,\cdots,r$. Then
$M$ is necessarily connected.
\end{prop}
We close this section by proving the Lemma that was used in the proof
of part (iv) of Proposition \ref{prop:gen-2nd-var}.
\begin{lem}
\label{lem:var-vol}In the setting of Proposition \ref{prop:gen-2nd-var},
the second variation of volume of any distinct phase $\Omega_{j}$
is given by
\begin{equation}
\delta^{2}|\Omega_{j}|=\int_{M_{j}}(\mathrm{div}_{\mathbb{R}^{n}}w)w\cdot N_{\partial\Omega_{j}}\label{eq:A-19}
\end{equation}
In this equation $N_{\partial\Omega_{j}}$ is the unit outward normal
field of $\partial\Omega_{j}$ and $M_{j}$ denotes collectivelly
the interfacial part of $\partial\Omega_{j}$, i.e. $M_{j}=\partial\Omega_{j}\backslash\partial\Omega$.\end{lem}
\begin{proof}
Let $(\xi^{t})_{t\in I}$ be a variation of $\mathbb{R}^{N}$ and
$\Omega_{j}^{t}=\xi^{t}(\Omega_{j})$. Then
\[
|\Omega_{j}^{t}|=\int_{\xi^{t}(\Omega_{j})}dx=\int_{\Omega_{j}}J\xi^{t}(y)dy
\]
where $J\xi^{t}$ is the Jacobian of $\xi^{t}$. For the second variation
of this functional we have 
\[
\delta^{2}|\Omega|=\left.\frac{d^{2}}{dt^{2}}|\Omega^{t}|\right|_{t=0}=\int_{\Omega}\left.\frac{\partial^{2}}{\partial t^{2}}J\xi^{t}(y)\right|_{t=0}dy.
\]
Application of the rule of differentiation of determinants and straight-forward
manipulations give
\[
\left.\frac{\partial^{2}}{\partial t^{2}}J\xi^{t}(y)\right|_{t=0}=\mathrm{div}_{\mathbb{R}^{N}}a+\frac{\partial w^{\alpha}}{\partial x^{\alpha}}\frac{\partial w^{\beta}}{\partial x^{\beta}}-\frac{\partial w^{\alpha}}{\partial x^{\beta}}\frac{\partial w^{\beta}}{\partial x^{\alpha}}.
\]
We are using greek indices for vector components and coordinates in
the surrounding space $\mathbb{R}^{N}$, and latin for submanifolds.
Summation convention is applicable to greek indices as well. Formula
(\ref{eq:A-19}) follows from this equality, the identity
\begin{alignat*}{1}
\frac{\partial w^{\alpha}}{\partial x^{\alpha}}\frac{\partial w^{\beta}}{\partial x^{\beta}}-\frac{\partial w^{\alpha}}{\partial x^{\beta}}\frac{\partial w^{\beta}}{\partial x^{\alpha}} & =\frac{\partial}{\partial x^{\alpha}}\left(w^{\alpha}\frac{\partial w^{\beta}}{\partial x^{\beta}}-w^{\beta}\frac{\partial w^{\alpha}}{\partial x^{\beta}}\right)\\
 & =\mathrm{div}_{\mathbb{R}^{n}}\left((\mathrm{div}_{\mathbb{R}^{n}}w)w-D_{w}w\right),
\end{alignat*}
Gauss's theorem in $\mathbb{R}^{N}$ and $a=D_{w}w$. Since the variation
preserves $\partial\Omega$, the integral over $\partial\Omega_{j}\backslash M_{j}$
drops. 
\end{proof}

\section{\label{sec:Spectr-Anal}Spectral Analysis of the 2nd Variation of
Area Form}

To keep the length of formulas to a minimum and focus on the essence
of the argument, we present the details for the two phase partitioning
problem and then indicate how the results generalize to more phases.

\subsection{Two phase partitioning problem}

Let $M$ be the interface of a two phase partitioning problem in $\Omega$,
which is assumed minimal, i.e. $\delta A(M)=0$. For linearized stability
we naturally study the minimal eigenvalue of the bilinear form
\begin{equation}
J(f)=\int_{M}\left(|\mathrm{grad}_{M}\, f\,|^{2}-|B_{M}|^{2}f^{2}\right)-\int_{\partial M}\sigma\, f^{2}\label{eq:J}
\end{equation}
For brevity we will write $\nabla^{M}f$ in place of $\mathrm{grad}_{M}\, f$.
Although $J$ and $\delta^{2}A^{\star}(M)$ are identical expressions,
we consider them from a different point of view: for the purposes
of spectral analysis $M$ is \emph{fixed} and $J$ is a nonlinear
functional on a properly defined functional space on $M$ containing
the admissible variations of $M$; thus its elements $f$ satisfy
the conditions of volume constancy
\begin{equation}
\int_{M}f=0\label{eq:vol-const}
\end{equation}
and the normalization condition
\begin{equation}
\int_{M}f^{2}=1.\label{eq:normal}
\end{equation}
As a matter of convenience, we introduce Lagrange multipliers and
the corresponding modified functional
\[
J^{\star}(f;\lambda,\mu)=\int_{M}\left(|\nabla^{M}f\,|^{2}-|B_{M}|^{2}f^{2}\right)-\int_{\partial M}\sigma\, f^{2}-\lambda\int_{M}f-\mu\int_{M}f^{2}
\]
and we are interested in the critical points of $J^{\star}$ or $J$
with the conditions (\ref{eq:vol-const}) and (\ref{eq:normal}).
\begin{prop}
\label{prop:eigv-prob}A necessary and sufficient condition for a
$C^{2}$ function $f$ on $M$ to be a critical point of $J^{\star}$,
or equivalently $J$ with the conditions (\ref{eq:vol-const}) and
(\ref{eq:normal}), is that it satisfies the following inhomogeneous
PDE with Neumann boundary condition:
\begin{equation}
\Delta_{M}f+(\mu+|B_{M}|^{2})f=-\frac{\lambda}{2}\label{eq:eigv-PDE}
\end{equation}
\begin{equation}
D_{\nu}f=\sigma f\quad\textrm{on}\,\partial M\label{eq:eigv-BC}
\end{equation}
\end{prop}
\begin{rem}
In (\ref{eq:eigv-PDE}) $\Delta_{M}$ is the Laplace-Beltrami operator
on $M$ defined by
\[
\Delta_{M}f=\mathrm{div}_{M}(\nabla^{M}f)=g^{-1/2}(g^{1/2}g^{ij}f_{,j})_{,i}
\]
in a local coordinate system $q^{1},\cdots,q^{N-1}$, where $g=\det(g_{ij})$,
$g_{ij}$ is the metric tensor and the comma operator denotes partial
derivative in the respective coordinate, i.e. $f_{,i}=\frac{\partial f}{\partial q^{i}}=D_{E_{i}}f$.
The summation convention on pairs of identical indices is assumed
throughout this paper. As $M$ is fixed, $g_{ij}$ is fixed and (\ref{eq:eigv-PDE})
is a linear equation.
\end{rem}

\begin{rem}
The $C^{2}$ condition on $f$ can be relaxed by considering the weak
form of (\ref{eq:eigv-PDE}), (\ref{eq:eigv-BC}).\end{rem}
\begin{proof}
The first variation of $J^{\star}$ is given by
\[
\begin{alignedat}{1}\delta J^{\star}(f)\phi & =\left.\frac{d}{dt}J^{\star}(f+t\phi)\right|_{t=0}\\
 & =2\int_{M}\left(\nabla^{M}f\cdot\nabla^{M}\phi-|B_{M}|^{2}f\phi\right)-2\int_{\partial M}\sigma f\phi-2\mu\int_{M}f\phi-\lambda\int_{M}\phi
\end{alignedat}
\]
By Green's formula for manifolds we obtain
\[
\delta J^{\star}(f)\phi=-2\int_{M}\left(\Delta_{M}f+|B_{M}|^{2}f+\mu f+\tfrac{1}{2}\lambda\right)\phi+2\int_{\partial M}(\nabla^{M}f\cdot\nu-\sigma f)\phi.
\]
When $\phi$ is a $C^{\infty}$ function with compact support in the
interior of $M$, the second integral on the right side drops and
by the fundamental lemma of the calculus of variations we obtain (\ref{eq:eigv-PDE}).
When the support of $\phi$ intersects the boundary of $M$, $\delta J^{\star}(f)\phi=2\int_{\partial M}(\nabla^{M}f\cdot\nu-\sigma f)\phi$
and, again by the same argumentation, we obtain (\ref{eq:eigv-BC}),
in view of the identity $D_{\nu}f=\nabla^{M}f\cdot\nu$. The converse
is trivial.
\end{proof}
Two are the relevant problems: (a) given a partitioning, to show that
it is unstable, i.e. to find \emph{a particular} admissible variation
$f$ such that $J(f)<0$, and (b) to prove that a partitioning $M$
is stable, i.e. \emph{for any} admissible variation $f\neq0$ we have
$J(f)>0$. The proposition next shows that for problems of the first
category it suffices to find a negative eigenvalue $\mu<0$ of problem
(\ref{eq:eigv-PDE}).
\begin{prop}
\label{prop:Categ-a}Let $M$ be a minimal two phase partitioning
and $f$ an eigenfunction of problem (\ref{eq:eigv-PDE}), (\ref{eq:eigv-BC})
with corresponding eigenvalue $\mu$. Then
\begin{equation}
J(f)=\mu.\label{eq:J-min}
\end{equation}
In particular, if $\mu<0$, $M$ is unstable.\end{prop}
\begin{rem}
Proposition \ref{prop:Categ-a} implies that no lower bound is necessary
for the functional $J$, in order to conclude that a minimal partitioning
is unstable, when a negative eigenvalue is at hand. This is in contrast
with problems of category (b).\end{rem}
\begin{proof}
Multiplication of (\ref{eq:eigv-PDE}) by $f$ and integration over
$M$ gives in view of (\ref{eq:vol-const}) and (\ref{eq:normal}):
\[
\int_{M}f\Delta_{M}f+\mu+\int_{M}|B_{M}|^{2}f^{2}=0
\]
Application of Green's formula on the first integral gives
\[
-\int_{M}|\nabla^{M}f|^{2}+\int_{\partial M}f\nu\cdot\nabla^{M}f+\mu+\int_{M}|B_{M}|^{2}f^{2}=0.
\]
By (\ref{eq:J}) and (\ref{eq:eigv-BC}) we obtain $J(f)=\mu$. The
second assertion follows trivially from this.
\end{proof}
In class (b) we need to know in advance that the functional $J$ has
a minimum under the conditions (\ref{eq:vol-const}) and (\ref{eq:normal}).
The difficulty is that the boundary integral $\int_{\partial M}f^{2}$
cannot be bounded above by $\int_{M}f^{2}$. However, if $f\in W^{1,2}(M)\equiv H^{1}(M)$,
by the boundary trace embedding theorem (\citep{key-18} §§ 5.34-5.37,
pp 163-166; \citep{key-17-1} § 8, pp 120-132) we have
\begin{equation}
\int_{\partial M}f^{2}\leqslant c_{BT}^{2}\left(\int_{M}|\nabla^{M}f\,|^{2}+\int_{M}f^{2}\right)\label{eq:BT-std}
\end{equation}
where $c_{BT}$ is a constant depending on $M$. Now using this estimate
for the boundary integral and the estimates $|B_{M}|^{2}\leqslant b_{0}^{2}$,
$\sigma\leqslant\sigma_{0}$, $b_{0}$ and $\sigma_{0}$ being certain
constants depending on $M$ and $\Sigma$, we obtain
\begin{equation}
\begin{alignedat}{1}J(f) & \geqslant\int_{M}|\nabla^{M}f\,|^{2}-b_{0}^{2}\int_{M}f^{2}-\sigma_{0}c_{BT}^{2}\left(\int_{M}|\nabla^{M}f\,|^{2}+\int_{M}f^{2}\right)\\
 & =(1-\sigma_{0}c_{BT}^{2})\int_{M}|\nabla^{M}f\,|^{2}-(b_{0}^{2}+\sigma_{0}c_{BT}^{2})\int_{M}f^{2}
\end{alignedat}
\label{eq:J-est-orig}
\end{equation}
from which we can conclude coercivity of $J$ if $\sigma_{0}<1/c_{BT}^{2}$.
In this way we can prove (see \citep{key-19}, Theorem 1.2, p.4) that
for sufficiently small principal curvatures of $\Sigma$ \emph{in
a neighborhood of} $\partial M$, $J$ has a minimum, which is necessarily
a critical point of $J^{\star}$, hence an eigenfunction of (\ref{eq:eigv-PDE})
with BC (\ref{eq:eigv-BC}). As a consequence, by (\ref{eq:J-min}),
if $J^{\star}$ has no non-positive eigenvalues, we can draw the conclusion
that $M$ is stable.

We can drop the hypothesis of sufficiently small principal curvatures
of $\Sigma$ in a neighborhood of\emph{ }$\partial M$ by replacing
(\ref{eq:BT-std}) by an interpolation estimate. The standard notation
for Sobolev spaces is used: $|u|_{L^{2}(M)}=(\int_{M}u^{2})^{1/2}$,
$|u|_{H^{1}(M)}=(|u|_{L^{2}(M)}^{2}+|\nabla^{M}u|_{L^{2}(M)}^{2})^{1/2}$,
$|u|_{L^{2}(\partial M)}=(\int_{\partial M}u^{2})^{1/2}$ are the
standard norms of $L^{2}(M)$, $H^{1}(M)$ and $L^{2}(\partial M)$.
\begin{lem}
\label{lem:interp-est}Let $M$ be a bounded $C^{2,\alpha}$ submanifold
of $\mathbb{R}^{N}$ with boundary. Then for every $\epsilon>0$ there
is a constant $c_{\epsilon}$ such that for any $u\in H^{1}(M)$
\begin{equation}
|u|_{L^{2}(\partial M)}\leqslant\epsilon|u|_{H^{1}(M)}+c_{\epsilon}|u|_{L^{2}(M)}\label{eq:BT-new}
\end{equation}
\end{lem}
\begin{proof}
We prove the estimate by contradiction. Assume there is a $\epsilon_{0}>0$
for which (\ref{eq:BT-new}) does not hold. Then for each $n\in\mathbb{N}$
there is a $u_{n}\in H^{1}(M)$ such that $|u_{n}|_{H^{1}(M)}=1$
and
\begin{equation}
|u_{n}|_{L^{2}(\partial M)}>\epsilon_{0}+n|u_{n}|_{L^{2}(M)}\label{eq:BT-new-hyp}
\end{equation}
By (\ref{eq:BT-new-hyp}) and (\ref{eq:BT-std}) we obtain
\[
|u_{n}|_{L^{2}(M)}<\frac{c_{BT}-\epsilon_{0}}{n}\to0,\; n\to\infty.
\]
Since $(u_{n})$ is bounded in $H^{1}(M)$, we have, possibly by passing
to a subsequence, $u_{n}\overset{H^{1}}{\rightharpoonup}u$. By $u_{n}\to0$
in $L^{2}(M)$ we easily conclude $u=0$. From the compactness of
the embedding $H^{1}(M)\hookrightarrow L^{2}(\partial M)$ (see Adams
\citep{key-18} §§ 5.34-5.37, pp 163-166 for the flat case) it follows
that $u_{n}\to0$ in $L^{2}(\partial M)$, which contradicts (\ref{eq:BT-new-hyp}).
\end{proof}
As an example we prove (\ref{eq:BT-new}) directly for a bounded hypersurface
$M$ of $\mathbb{R}^{N}$ with boundary.
\begin{example}
\label{ex:estim-direct}Assume that there is a $x_{0}\in\mathbb{R}^{N}$
such that $x_{0}\cdot\nu(p)>0$ for all $p\in\partial M$. Without
loss of generality set $x_{0}=0$. For $u\in C^{\infty}(M)$, $x$
being the position vector in $\mathbb{R}^{N}$, by (\ref{eq:div-formula})
\[
\int_{M}\mathrm{div}_{M}(xu^{2})=-\int_{M}H\cdot xu^{2}+\int_{\partial M}(x\cdot\nu)u^{2}
\]
we obtain
\begin{equation}
\begin{alignedat}{1}\int_{\partial M}(x\cdot\nu)u^{2} & =\int_{M}H\cdot xu^{2}+\int_{M}(u^{2}\mathrm{div}_{M}x+2ux\cdot\nabla^{M}u)\\
 & \leqslant\frac{1}{2}\int_{M}(H\cdot x)^{2}u^{2}+\frac{1}{2}\int_{M}u^{2}+\int_{M}u^{2}\mathrm{div}_{M}x\\
 & +\frac{1}{\epsilon^{2}}\int_{M}|x|^{2}u^{2}+\epsilon^{2}\int_{M}|\nabla^{M}u|^{2}
\end{alignedat}
\label{eq:ex-2}
\end{equation}
By the compactness of $\partial M$ and $\overline{M}$ there are
positive constants $c_{0}$, $c_{1}$, $c_{2}$ such that $x\cdot\nu\geqslant c_{0}$,
$|x|^{2}\leqslant c_{1}$ and $|H|\leqslant c_{2}$. To compute $\mathrm{div}_{M}x$
apply the definition of operator $\mathrm{div}_{M}$ (Simon \citep{key-16})
in a chart with basis vector fields $E_{1},\cdots,E_{N-1}$: 
\[
\mathrm{div}_{M}x=g^{ij}\left\langle D_{E_{i}}x,E_{j}\right\rangle =g^{ij}\left\langle E_{i},E_{j}\right\rangle =g^{ij}g_{ij}=N-1
\]
By (\ref{eq:ex-2}) we obtain
\[
c_{0}|u|_{L^{2}(\partial M)}^{2}\leqslant\epsilon^{2}|u|_{H^{1}(M)}^{2}+\left(N-\frac{1}{2}+\frac{1}{2}c_{1}c_{2}^{2}+\frac{1}{\epsilon^{2}}\right)|u|_{L^{2}(M)}^{2}.
\]
By a density argument we extend to $u\in H^{1}(M)$. \hfill{} $\square$
\end{example}
Estimate (\ref{eq:BT-new}) makes it possible to select $\epsilon>0$
so small that the coefficient of $\int_{M}|\nabla^{M}f\,|^{2}$ in
(\ref{eq:J-est-orig}) becomes positive. The following unequality
replaces (\ref{eq:J-est-orig}):
\begin{equation}
J(f)\geqslant(1-2\sigma_{0}\epsilon^{2})|u|_{H^{1}(M)}^{2}-(b_{0}^{2}+1+2\sigma_{0}c_{\epsilon}^{2})|u|_{L^{2}(M)}^{2}.\label{eq:J-est-new}
\end{equation}
Now we can conveniently prove the main result for problems of class
(b):
\begin{prop}
\label{prop:Categ-b}Let $M$ be a minimal two phase partitioning
in $\Omega\subset\mathbb{R}^{N}$. Then for any $f\in H^{1}(M)$ satisfying
(\ref{eq:vol-const}), (\ref{eq:normal}) we have 
\begin{equation}
J(f)\geqslant\mu_{1}\label{eq:J-min-eigv}
\end{equation}
where $\mu_{1}$ is the smallest eigenvalue of problem (\ref{eq:eigv-PDE}),
(\ref{eq:eigv-BC}). In particular, if $\mu_{1}>0$, $M$ is stable.\end{prop}
\begin{proof}
Let $X=\{u\in H^{1}(M):|u|_{L^{2}(M)}\leqslant1,\,\int_{M}u=0\}$.
By the continuity of the $L^{2}(M)$-norm and $u\mapsto\int_{M}u$
as mappings $H^{1}(M)\to\mathbb{R}$, it follows that $X$ is a closed
subset of $H^{1}(M)$. The convexity of $X$ is clear. Hence $X$
is a weakly closed subset of $H^{1}(M)$. By (\ref{eq:J-est-new})
it follows immediately that $J$ is coercive on $X$. The sequential
weakly lower semicontinuity of $J$ follows from the sequential weakly
lower semicontinuity of the norm of $H^{1}(M)$ and the compactness
of the embedding $H^{1}(M)\hookrightarrow L^{2}(M)$: $u_{n}\overset{H^{1}}{\rightharpoonup}u$
implies $u_{n}\overset{L^{2}}{\to}u$. The conditions in \citep{key-19}
Theorem 1.2 are satisfied and by this we conclude that $J$ attains
its infimum in $X$. The position of the infimum is a critical point
of $J^{\star}$ and, as it was shown in Proposition \ref{prop:Categ-a},
it is a solution of equation (\ref{eq:eigv-PDE}) with BC (\ref{eq:eigv-BC}).
Inequality (\ref{eq:J-min-eigv}) follows trivially form this.
\end{proof}

\subsection{Three and more phases}

Since disconnected partitionings in strictly convex sets are always
unstable according to Proposition \ref{prop:SZ-gen}, we need only
consider connected partitionings. The functional $J$ for the $(m+1)$-phase
partitioning problem reads:
\begin{equation}
J(f_{1},\dots,f_{m})=\sum_{i=1}^{m}\gamma_{i}\int_{M_{i}}\left(|\mathrm{grad}_{M_{i}}\, f_{i}|^{2}-|B_{M_{i}}|^{2}f_{i}^{2}\right)-\sum_{i=1}^{m}\gamma_{i}\int_{\partial M_{i}}\sigma_{i}f_{i}^{2}\label{eq:J-gen}
\end{equation}
The volume constraints are
\begin{equation}
\int_{M_{i}}f_{i}=0,\; i=1,\cdots,m\label{eq:vol-const-gen}
\end{equation}
the normalization condition is
\begin{equation}
\sum_{i=1}^{m}\int_{M_{i}}f_{i}^{2}=1\label{eq:normal-gen}
\end{equation}
and the corresponding modified functional is
\[
\begin{alignedat}{1}J^{\star}(f_{1},\cdots,f_{m};\lambda_{1},\cdots,\lambda_{m},\mu) & =\sum_{i=1}^{m}\gamma_{i}\left[\int_{M_{i}}\left(|\nabla^{M_{i}}f_{i}|^{2}-|B_{M_{i}}|^{2}f_{i}^{2}\right)-\int_{\partial M_{i}}\sigma_{i}f_{i}^{2}\right]\\
 & -\sum_{i=1}^{m}\lambda_{i}\int_{M_{i}}f_{i}-\mu\sum_{i=1}^{m}\int_{M_{i}}f_{i}^{2}
\end{alignedat}
\]

Proposition \ref{prop:eigv-prob} extends without difficulty to connected
multiphase problems, with (\ref{eq:eigv-PDE}) holding on each inteface
in the following form:
\begin{equation}
\gamma_{i}\Delta_{M_{i}}f_{i}+(\mu+\gamma_{i}|B_{M_{i}}|^{2})f_{i}=-\frac{\lambda_{i}}{2}\label{eq:eigv-PDE-gen}
\end{equation}
The boundary conditions retain the form of (\ref{eq:eigv-BC}) for
each $i$. Propositions \ref{prop:Categ-a}, \ref{prop:Categ-b} hold
as they are. In the proof of Proposition \ref{prop:Categ-b} the considered
space is 
\[
X=\{(u_{1},\cdots,u_{m})\in H^{1}(M):|u|_{L^{2}(M)}\leqslant1,\,\int_{M_{i}}u_{i}=0,\, i=1,\cdots,m\}
\]
Here we have the following definitions: $H^{1}(M)=H^{1}(M_{1})\times\cdots\times H^{1}(M_{m})$,
$L^{2}(M)=L^{2}(M_{1})\times\cdots\times L^{2}(M_{m})$ and $|u|_{L^{2}(M)}^{2}=\sum_{i=1}^{m}|u|_{L^{2}(M_{i})}^{2}$
as usually.

The proposition next characterizes all connected $m$-phase partitionings
by stating that each $m$-phase partitioning problem reduces to $m-1$
independent 2-phase problems.
\begin{prop}
\label{prop:reduc}Let $\Omega$ be an open, bounded subset of $\mathbb{R}^{N}$
with $C^{2}$ boundary and $M=(M_{i})_{i=1}^{m-1}$ a connected minimal
$m$-phase partitioning of $\Omega$ with volume constancy constraint.
Then $M$ is stable if and only if each $M_{i}$ $(i=1,\cdots,m-1)$
is stable. $M$ is unstable if and only if there is at least one unstable
$M_{i}$ $(i=1,\cdots,m-1)$.
\end{prop}
The proof is straight-forward and follows by the fact that the $M_{i}$
appear individually in the form of the second variation of area $J$
and the constraints.

Examples for stable and unstable multiphase partitionings are easily
constructed from 2-phase partitionings by means of Proposition \ref{prop:reduc}.
Therefore our applications focus on 2-phase partitionings.

\section{\label{sec:Apps-RN}Some Applications to Partitioning Problems in
$\mathbb{R}^{N}$}

As applications of the spectral analysis of the 2nd variation of area,
we derive in this section some general conclusions concerning partitionings
in $\mathbb{R}^{N}$.

\subsection{Stability of $N$-dimensional spherical partitionings}

The stability of partitionings in which one phase has the shape of
a sphere, $M=\mathbb{S}^{N-1}$, has a direct physical meaning. It
is the basis for modeling the stability of emulsions, i.e. suspensions
of small liquid droplets or deformable solid particles in a surrounding
fluid. As $\partial M=\emptyset$, all boundary integrals are absent,
in particular the boundary condition (\ref{eq:eigv-BC}). In 3-dimensions
one may directly proceed to the solution of (\ref{eq:eigv-PDE}) in
spherical polar coordinates using periodic conditions in place of
(\ref{eq:eigv-BC}). The spectral theory of the operator $\Delta_{M}$
is well-known (see for example \citep{key-20}, Ch. V, §8, p. 314;
Ch. VII, §5, pp 510-512 and \citep{key-21} §IV.2). Its eigenvalues
are given by $\lambda_{l}=-l(l+1)$, $l=0,1,\cdots$ and the corresponding
eigenvectors are the spherical harmonics $Y_{l}^{m}$, $m=0,\pm1,\cdots\pm l$,
the first of which are
\[
Y_{0}^{0}(\theta,\phi)=\frac{1}{2\sqrt{\pi}},\, Y_{1}^{0}(\theta,\phi)=\frac{3}{2\sqrt{6\pi}}\cos\theta,\, Y_{1}^{\pm1}(\theta,\phi)=\pm\frac{3}{2\sqrt{6\pi}}\sin\theta e^{-i\phi}
\]
Since the eigenvectors of $J$ satisfy the volume constancy condition
(\ref{eq:vol-const}), the first of these is not admissible. Integration
of (\ref{eq:eigv-PDE}) in view of (\ref{eq:vol-const}) gives $\lambda=0$.
Thus the minimal eigenvalue of $J$, as obtained by $l=1$, is $\mu_{1}=l(l+1)-(N-1)=0$,
which whould imply neutral stability. A more careful examination of
these eigenvectors reveals that they are not true variations, but
translations along the axes of the coordinate system. On discarding
them and procceding to the next available eigenvalue, $l=2$, we have
\[
\mu_{1}=l(l+1)-(N-1)=4>0
\]
which implies stabilty. For $N>3$ (see \citep{key-21}) $\lambda_{l}=-l(l+N-2)$,
$l=0,1,\cdots$ and again $\mu_{1}=0$. Discarding translations again
yields stability. In the next section we prove that the same situation
pertains also to the 2-dimensional case. For the following proposition,
no regularity and convexity conditions are necessary for $\Omega$.
\begin{prop}
Let $\Omega\subset\mathbb{R}^{N}$ be an open set and $\Omega_{1}=B(x_{0},R)$
a ball such that $\overline{\Omega}_{1}\subset\Omega$. The two-phase
partitioning of $\Omega$ defined by $M=\partial\Omega_{1}$ is stable.
\end{prop}
Note that balls are \emph{never} absolute minimizers of partitioning
problems in bounded sets. This is most conveniently seen by moving
the ball until it makes contact with the boundary, $M\cap\partial\Omega\neq\emptyset$.
Then it is clear that the translated ball is not even minimal, since
at the contact point the normal of $M$ is not tangent to $\partial\Omega$.
This implies the existence of a variation which decreases the ball's
area and this proves that the original ball is not an absolute minimizer.

\subsection{Existence of unstable two phase partitionings}

The existence of stable partitionings is guaranteed by the existence
of absolute minimizers (see \cite{key-23} and \citep{key-19} Th.
1.4, p. 6). Here we give a proof of existence of unstable partitionings
in $\mathbb{R}^{N}$ under certain conditions on $\Omega$. Let $M$
be a minimal two phase partitioning of $\Omega$. We assume $\Omega$
is convex and has the following property: for all sufficiently large
$\sigma$ there is a convex set $\Omega_{\sigma}$ containing $M$
and satisfying the condition
\begin{lyxlist}{00.00.0000}
\item [{(H)}] $\Omega_{\sigma}$ contacts $\Sigma=\partial\Omega$ along
$\partial M$, i.e. $T_{p}\partial\Omega_{\sigma}=T_{p}\Sigma$ at
all $p\in\partial M$, and
\[
II_{\partial\Omega_{\sigma},p}(N_{p},N_{p})\geqslant\sigma,\; p\in\partial M
\]
($N_{p}$ is the normal field of $M$ at $p$.)\end{lyxlist}
\begin{example}
(i) Let $\Omega=B(0,\frac{1}{\sigma_{0}})\subset\mathbb{R}^{2}$,
and $D$ a circle intersecting $\partial\Omega$ at right angles.
Further let $M$ be the circular arc obtained by the intersection
of $\Omega$ and $D$, and $M\cap\partial\Omega=\{p_{1},p_{2}\}$.
If $C_{1}$, $C_{2}$ be the circles of radius $\frac{1}{\sigma}$
contained in $\Omega$ and contacting $\Sigma=\partial\Omega$ at
$p_{1}$, $p_{2}$ respectively, then the convex hull of $M\cup C_{1}\cup C_{2}=\Omega_{\sigma}$
satisfies (H).

(ii) Consider a rhombus and a circular arc with its center positioned
at one of its vertices. Let $\Omega$ be the solid obtained by the
revolution of the rhombus about the axis of the rombus that passes
through the center of the circular arc. Let $M$ be the surface obtained
by the revolution of the arc. By elementary geometric arguments similar
to (i) it is easily established that for all sufficiently large $\sigma$
there is a convex set $\Omega_{\sigma}$ containing $M$ and satisfying
condition (H).\end{example}
\begin{prop}
\label{prop:exist-unst-part}Let $M$ be a minimal partitioning of
$\Omega$ in $\mathbb{R}^{N}$. Assume that $\Omega$ satisfies condition
(H). Then there is a convex set $\Omega^{\star}$ for which $M$ is
an unstable partitioning.\end{prop}
\begin{proof}
Assume $M$ is stable or neutrally stable, for otherwise there is
nothing to prove. Let
\[
\mu_{1}=\underset{\int_{M}f=0,\,|f|_{L^{2}(M)}=1}{\inf}J(f)\quad\geqslant0
\]
Let $\epsilon>0$ small and $f$ be a variation of $M$ satisfying
(\ref{eq:normal}), (\ref{eq:vol-const}) and such that $J(f)<\mu_{1}+\epsilon$.
We can assume that $f$ is not identically vanishing on $\partial M$,
for if $f=0$ on $\partial M$, select $f_{0}\in\mathbb{R}$, $f_{0}\neq0$
such that
\[
-\int_{M}|B_{M}|^{2}-\int_{\partial M}II(N,N)<\frac{2}{f_{0}}\int_{M}|B_{M}|^{2}f
\]
and then we have $f+f_{0}=f_{0}\neq0$ on $\partial M$ and
\[
\begin{alignedat}{1}J(f+f_{0}) & =J(f)-f_{0}^{2}\int_{M}|B_{M}|^{2}-2f_{0}\int_{M}|B_{M}|^{2}f-f_{0}^{2}\int_{\partial M}II(N,N)\\
 & <J(f).
\end{alignedat}
\]
Since $\Omega$, $M$ were assumed to satisfy (H) there is a $\sigma>0$
and $\Omega_{\sigma}=:\Omega^{\star}$ such that 
\[
\sigma>\frac{k}{|f|_{L^{2}(\partial M)}^{2}}+\underset{p\in\partial M}{\sup}II_{\partial\Omega,p}(N_{p},N_{p})
\]
and
\[
II_{\partial\Omega^{\star}}(N,N)\geqslant\sigma
\]
on $\partial M$. $k$ is a positive number to be fixed later. We
have
\[
\begin{alignedat}{1}J_{\Omega^{\star}}(f) & =\int_{M}\left(|\nabla^{M}f\,|^{2}-|B_{M}|^{2}f^{2}\right)-\int_{\partial M}II_{\partial\Omega^{\star}}(N,N)\, f^{2}\\
 & \leqslant\int_{M}\left(|\nabla^{M}f\,|^{2}-|B_{M}|^{2}f^{2}\right)-\sigma\int_{\partial M}f^{2}\\
 & \leqslant\int_{M}\left(|\nabla^{M}f\,|^{2}-|B_{M}|^{2}f^{2}\right)-\int_{\partial M}II_{\partial\Omega}(N,N)\, f^{2}-k\\
 & =J(f)-k\leqslant\mu_{1}+\epsilon-k
\end{alignedat}
\]
Choosing $k>\mu_{1}+\epsilon$ completes the proof.
\end{proof}

\section{\label{sec:Apps-2D}Application to 2-Dimensional Partitioning Problems}

Two-dimensional partitionings are particularly simple, for in this
case
\[
\Delta_{M}f=\frac{d^{2}f}{ds^{2}}
\]
where $s$ is the arc length of $M$ and the integrals over $\partial M$
reduce to numbers. The boundary condition (\ref{eq:eigv-BC}) reduces
to
\begin{equation}
\pm\frac{df}{ds}=\sigma f,\;\textrm{at}\, s=0,L\label{eq:BC-2d}
\end{equation}
$L=|M|$ being the length of $M$. The plus sign applies at $s=L$
and minus at $s=0$. We are using the values $\sigma=\sigma_{1}$
at $s=0$ and $\sigma=\sigma_{2}$ at $s=L$. From part (i) of Proposition
\ref{prop:gen-2nd-var} the curvature $\kappa$ is constant, thus
the only possibilities for $M$ are line segments and circular arcs.
The orientation of $M$ is selected so that $\kappa\geqslant0$. Further,
$|B_{M}|=\kappa=1/R$, where $R$ is the radius of the arc or $\infty$
for line segments.

For equation (\ref{eq:eigv-PDE}) we have the following three types
of solution, depending on the sign of $\mu+|B_{M}|^{2}$:
\[
\begin{aligned}(I) & \quad f(s)=-\frac{\lambda}{2k^{2}}+C\sin(ks)+D\cos(ks), & k^{2}=\mu+\kappa^{2}\Leftrightarrow\mu>-\kappa^{2}\\
(II) & \quad f(s)=\frac{\lambda}{2k^{2}}+Ce^{ks}+De^{-ks}, & k^{2}=-(\mu+\kappa^{2})\Leftrightarrow\mu<-\kappa^{2}\\
(III) & \quad f(s)=-\frac{\lambda}{4}s^{2}+Cs+D & \mu=-\kappa^{2}
\end{aligned}
\]
The constants $\lambda$, $C$, $D$ in each case are determined by
the two conditions of (\ref{eq:BC-2d}) and one of (\ref{eq:vol-const}),
which form a linear homogeneous system of three equations in these
three variables. The condition for existence of solutions of this
system is obtained by setting its determinant to 0, which gives a
nonlinear equation for $k$. With a solution for $k$ at hand, we
can determine the eigenvalue $\mu$ by the last column in the above
table of possible solutions for $f$ and the eivenvectors $(\lambda,C,D)$,
each determining an eigenfunction $f$ of problem (\ref{eq:eigv-PDE}).
Not all three cases (I)-(III) need to be considered, depending on
the problem under study.

\subsection{Case I: $-\kappa^{2}<\mu<0\Leftrightarrow0<\frac{k}{\kappa}<1$}

By (\ref{eq:BC-2d}) $s=0,L$ and (\ref{eq:vol-const}) we obtain
the system
\begin{equation}
\begin{aligned} & -\frac{\lambda}{2k^{2}}+\frac{k}{\sigma_{1}}C+D=0\\
 & -\frac{\lambda}{2k^{2}}+\left[\sin(kL)-\frac{k}{\sigma_{2}}\cos(kL)\right]C+\left[\cos(kL)+\frac{k}{\sigma_{2}}\sin(kL)\right]D=0\\
 & -\frac{\lambda L}{2k}+\left[1-\cos(kL)\right]C+\sin(kL)\, D=0
\end{aligned}
\label{eq:case-I}
\end{equation}
The case of stable and neutrally stable eigenvalues $\mu\geqslant0$
is contained here.

\subsection{Case II: $\mu<-\kappa^{2},\; k>0$}

By (\ref{eq:BC-2d}) $s=0,L$ and (\ref{eq:vol-const}) we obtain
the system
\begin{equation}
\begin{aligned} & \frac{\lambda}{2k^{2}}+\left(1+\frac{k}{\sigma_{1}}\right)C+\left(1-\frac{k}{\sigma_{1}}\right)D=0\\
 & \frac{\lambda}{2k^{2}}+\left(1-\frac{k}{\sigma_{2}}\right)e^{kL}C+\left(1+\frac{k}{\sigma_{2}}\right)e^{-kL}D=0\\
 & \frac{\lambda L}{2k}+\left(e^{kL}-1\right)C-\left(e^{-kL}-1\right)D=0
\end{aligned}
\label{eq:case-II}
\end{equation}

\subsection{Case III: $\mu=-\kappa^{2}$}

As previously we obtain the system
\begin{equation}
\begin{aligned} & C+\sigma_{1}D=0\\
 & \tfrac{1}{4}L\left(2-\sigma_{2}L\right)\lambda+\left(\sigma_{2}L-1\right)C+\sigma_{2}D=0\\
 & -\tfrac{1}{6}L^{2}\lambda+LC+2D=0
\end{aligned}
\label{eq:case-III}
\end{equation}

By the first of equations (\ref{eq:case-III}), $C=-\sigma_{1}D$,
and the remaining two equations give a system, the solvability of
which is equivalent to the equation
\begin{equation}
\sigma_{1}\sigma_{2}L^{2}-4(\sigma_{1}+\sigma_{2})L+12=0.\label{eq:IIIa}
\end{equation}
As a consequence, when it happens that the length of the interface
has one of the following two values
\[
L=2\frac{\sigma_{1}+\sigma_{2}\pm\sqrt{\sigma_{1}^{2}+\sigma_{2}^{2}-\sigma_{1}\sigma_{2}}}{\sigma_{1}\sigma_{2}}
\]
the partitioning is unstable. We will prove a more general result
in Proposition \ref{prop:crit-1}.

\subsection{Stability and instability criteria}

The following proposition generalizes the previous result.
\begin{prop}
\label{prop:crit-1}Let $\Omega$ be a bounded, convex, open subset
of $\mathbb{R}^{2}$ and $M$ a minimal two phase partitioning of
$\Omega$ with length $L$. Assume there is a neighborhood of the
points $\partial M$ in $\Sigma=\partial\Omega$ which is a $C^{2}$
curve and the curvatures of $\Sigma$ at these points are $\sigma_{1}$,
$\sigma_{2}$. If $L$ satisfies the condition
\begin{equation}
2\frac{\sigma_{1}+\sigma_{2}-\sqrt{\sigma_{1}^{2}+\sigma_{2}^{2}-\sigma_{1}\sigma_{2}}}{\sigma_{1}\sigma_{2}}\leqslant L\leqslant2\frac{\sigma_{1}+\sigma_{2}+\sqrt{\sigma_{1}^{2}+\sigma_{2}^{2}-\sigma_{1}\sigma_{2}}}{\sigma_{1}\sigma_{2}}\label{eq:crit-1}
\end{equation}
then $M$ is unstable.\end{prop}
\begin{proof}
We only have to prove the inequalities. Letting $B=\frac{\lambda}{2k^{2}}$,
$x=Lk$, $a=\sigma_{1}L$ and $b=\sigma_{2}L$, the solvability condition
for system (\ref{eq:case-II}) of case (II) (in the variables $B$,
$C$, $D$) is
\[
D(x)=\left|\begin{array}{ccc}
1 & 1+\frac{x}{a} & 1-\frac{x}{a}\\
1 & \left(1-\frac{x}{b}\right)e^{x} & \left(1+\frac{x}{b}\right)e^{-x}\\
x & e^{x}-1 & 1-e^{-x}
\end{array}\right|=0
\]
Performing operations we obtain
\[
D(x)=4\left[1+\frac{1}{2}\left(\frac{1}{a}+\frac{1}{b}\right)x^{2}\right]\cosh x-2\left(1+\frac{1}{a}+\frac{1}{b}+\frac{x^{2}}{ab}\right)x\sinh x-4.
\]
For $x\to+\infty$ it is easily established that $D(x)\to-\infty$.
At $x=0$ we have $D(0)=0$, so we expand $D$ in a power series of
$x$ about 0. To this purpose we consider the function
\[
f(x)=p(x)\cosh x+q(x)\sinh x
\]
and prove by induction on $n$ that
\begin{equation}
f^{(n)}(x)=\cosh x\sum_{i=0}^{n}\left(\begin{array}{c}
n\\
i
\end{array}\right)p_{\rho_{n-i}}^{(i)}+\sinh x\sum_{i=0}^{n}\left(\begin{array}{c}
n\\
i
\end{array}\right)q_{\rho_{n-i}}^{(i)}\label{eq:f(n)}
\end{equation}
where $\rho_{m}=m\mod2$ and $p_{0}=p$, $q_{0}=q$, $p_{1}=q$, $q_{1}=p$.
Assuming the validity of (\ref{eq:f(n)}) for $n$ we will prove its
validity for $n+1$. Clearly all $f^{(i)}$ have the same form with
$f$. The coefficient of $\cosh x$ in $f^{(n+1)}$ is a sum of numerical
multiples of the derivatives $p^{(n+1)},\cdots,p^{(0)}=p$ and $q^{(n+1)},\cdots,q^{(0)}=q$
and it is obtained by taking the derivative of $f^{(n)}$ in (\ref{eq:f(n)}).
We have
\[
\begin{gathered}\sum_{i=0}^{n}\left(\begin{array}{c}
n\\
i
\end{array}\right)p_{\rho_{n-i}}^{(i+1)}+\sum_{i=0}^{n}\left(\begin{array}{c}
n\\
i
\end{array}\right)q_{\rho_{n-i}}^{(i)}=\\
p_{0}^{(n+1)}+\sum_{i=0}^{n-1}\left(\begin{array}{c}
n\\
i
\end{array}\right)p_{\rho_{n-i}}^{(i+1)}+\sum_{i=1}^{n}\left(\begin{array}{c}
n\\
i
\end{array}\right)q_{\rho_{n-i}}^{(i)}+q_{\rho_{n}}^{(0)}=\\
p_{0}^{(n+1)}+\sum_{i=0}^{n-1}\left(\begin{array}{c}
n\\
i
\end{array}\right)p_{\rho_{n-i}}^{(i+1)}+\sum_{j=0}^{n-1}\left(\begin{array}{c}
n\\
j+1
\end{array}\right)p_{\rho_{n-j}}^{(j+1)}+p_{\rho_{n+1}}^{(0)}=\\
p_{0}^{(n+1)}+\sum_{i=0}^{n-1}\left[\left(\begin{array}{c}
n\\
i
\end{array}\right)+\left(\begin{array}{c}
n\\
i+1
\end{array}\right)\right]p_{\rho_{n-i}}^{(i+1)}+p_{\rho_{n+1}}^{(0)}
\end{gathered}
\]
On the second equality we used the change of summation index $j=i-1$
and the identity $q_{\rho_{n-j-1}}^{(j+1)}=p_{\rho_{n-j}}^{(j+1)}$.
Application of the binomial theorem and a second redefinition of the
sumation index by $j=i+1$, prove the assertion.

By (\ref{eq:f(n)}) with $p(x)=4\left[1+\frac{1}{2}\left(\frac{1}{a}+\frac{1}{b}\right)x^{2}\right]$,
$q(x)=-2\left(1+\frac{1}{a}+\frac{1}{b}+\frac{x^{2}}{ab}\right)x$
we obtain $f^{\prime}(0)=f^{\prime\prime}(0)=f^{\prime\prime\prime}(0)=0$
and
\[
f^{(4)}(0)=-\frac{4}{ab}(ab-4a-4b+12).
\]
The expansion of $D$ is
\[
D(x)=-\frac{1}{6ab}(ab-4a-4b+12)x^{4}+O(x^{5}).
\]
Thus $D(x)>0$ in a neighborhood $]0,\delta[$, $\delta>0$, when
$ab-4a-4b+12<0$. Using the definitions of $a$, $b$ in this inequality
we obtain
\[
\sigma_{1}\sigma_{2}L^{2}-4(\sigma_{1}+\sigma_{2})L+12<0
\]
the solution of which is (\ref{eq:crit-1}). By the continuity of
$D$ and the intermediate value theorem of calculus we conclude the
existence of a positive root of $D(x)=0$, which completes the proof.
\end{proof}
For $\sigma_{1}=\sigma_{2}=\sigma$, by (\ref{eq:crit-1}) we have
instability when $\frac{2}{\sigma}\leqslant L\leqslant\frac{6}{\sigma}$
and $L\to\infty$ as $\sigma\to0$, which suggests that for flat boundaries
all minimal partitionings are stable. This is conveniently proved
by solving systems (\ref{eq:case-I}), (\ref{eq:case-II}) and (\ref{eq:case-III}):
\begin{prop}
Let $\Omega$ and $M$ satisfy the conditions of Proposition \ref{prop:crit-1}
and $\sigma_{1}=\sigma_{2}=0$. Then $M$ is stable.\end{prop}
\begin{proof}
When $\sigma_{1}=\sigma_{2}=0$, by (\ref{eq:IIIa}) it is clear that
(\ref{eq:case-III}) has no solution and it is easily checked that
(\ref{eq:case-II}) has only the trivial solution. System (\ref{eq:case-I})
reduces to
\[
\lambda=C=0,\;\sin(kL)\, D=0
\]
which has nontrivial solutions only when $\sin(kL)=0$, i.e. $kL=n\pi$,
$n\in\mathbb{N}$. By the restriction $0<\frac{k}{\kappa}<1$ for
case (I) it follows that $\kappa L>n\pi$. As $\omega=\kappa L$ is
the angle of the circular sector defined by $M$ and the tangents
to $\partial\Omega$ at the extremities of $M$, by the convexity
of $\Omega$ we have $\omega<\pi$. The case $k=\kappa$ correspomding
to $\mu=0$ gives also no eigenvalues. Thus we have only positive
eigenvalues by case (I), which are given by $k_{n}=\frac{n\pi}{L}$
or 
\[
\mu_{n}=\frac{n^{2}\pi^{2}}{L^{2}}-\kappa^{2}>0
\]
and this according to Proposition \ref{prop:Categ-b} implies stability.
\end{proof}
As a next application of equations (\ref{eq:case-I})-(\ref{eq:case-III})
we give a proof of the stability of circles. This is essentially a
variational proof of the well-known fact that among all 2-dimensional
geometric shapes of equal area, circles have least perimeter.
\begin{prop}
Let $\Omega\subset\mathbb{R}^{2}$ be an open set and $\Omega_{1}=B(x_{0},R)$
such that $\overline{\Omega}_{1}\subset\Omega$. Then $M=\partial\Omega_{1}$
is a stable two phase partitioning of $\Omega$.\end{prop}
\begin{proof}
Integration of (\ref{eq:eigv-PDE}) in view of (\ref{eq:vol-const})
gives $\lambda=0$. From systems (\ref{eq:case-I})-(\ref{eq:case-III})
only the third equation remains in place. Additionally we have the
periodicity condition $f(0)=f(\frac{2\pi}{\kappa})$. It is easily
checked that cases (II) and (III) have only trivial solutions, while
for (I) we have
\[
\begin{aligned}C\sin\frac{2\pi k}{\kappa}+D\left(\cos\frac{2\pi k}{\kappa}-1\right)=0\\
\left(1-\cos\frac{2\pi k}{\kappa}\right)C+D\sin\frac{2\pi k}{\kappa}=0
\end{aligned}
\]
which has nontrivial solutions if and only if $\cos\frac{2\pi k}{\kappa}=1$,
i.e. $k=\kappa$ which corresponds to neutral stability as in this
case $\mu=0$. The eigenvectors are
\[
f(s)=C\sin(\kappa s)+D\cos(\kappa s)
\]
i.e. linear combinations of $f_{1}(s)=\sin(\kappa s)$ and $f_{2}(s)=\cos(\kappa s)$
expressing translations along the $y$ and $x$ axis respectively.
Since they are essentially not true variations of $M$, we can discard
them and thus there are only positive eigenvalues, which proves the
assertion.
\end{proof}
Partitionings with large interfaces are unstable. More precisely:
\begin{prop}
\label{prop:crit-2}Let $\Omega$ and $M$ satisfy the conditions
of Proposition \ref{prop:crit-1}. If $L$ satisfies the condition
\begin{equation}
L>2\frac{\sigma_{1}+\sigma_{2}+\sqrt{\sigma_{1}^{2}+\sigma_{2}^{2}-\sigma_{1}\sigma_{2}}}{\sigma_{1}\sigma_{2}}\label{eq:crit-2}
\end{equation}
then $M$ is unstable.\end{prop}
\begin{proof}
Integration of (\ref{eq:eigv-PDE}), taking into account (\ref{eq:eigv-BC})
and (\ref{eq:vol-const}), gives
\[
\int_{\partial M}\sigma f+\int_{M}|B_{M}|^{2}f=-\tfrac{1}{2}\lambda|M|,
\]
which in two dimensions simplifies to the equation
\begin{equation}
\sigma_{1}f(0)+\sigma_{2}f(L)=-\tfrac{1}{2}\lambda L\label{eq:crit-2-eq-1}
\end{equation}
By remarking that when $f$ is an eigenfunction of (\ref{eq:eigv-PDE})
then also $f^{\star}$ defined by $f^{\star}(s)=f(L-s)$ is an eigenfunction
for the same eigenvalue, we obtain the equation
\begin{equation}
\sigma_{2}f(0)+\sigma_{1}f(L)=-\tfrac{1}{2}\lambda L\label{eq:crit-2-eq-2}
\end{equation}
The solution of the system (\ref{eq:crit-2-eq-1}), (\ref{eq:crit-2-eq-2})
is
\[
f(0)=f(L)=-\frac{\lambda L}{2(\sigma_{1}+\sigma_{2})}
\]
By (II) we have $f(0)=\frac{\lambda}{2k^{2}}+C+D$ and $f(L)=\frac{\lambda}{2k^{2}}+Ce^{kL}+De^{-kL}$,
from which it follows that
\[
C=-\frac{\lambda}{2}\frac{1-e^{-kL}}{e^{kL}-e^{-kL}}\left(\frac{1}{k^{2}}+\frac{L}{\sigma_{1}+\sigma_{2}}\right),\; D=-\frac{\lambda}{2}\frac{e^{kL}-1}{e^{kL}-e^{-kL}}\left(\frac{1}{k^{2}}+\frac{L}{\sigma_{1}+\sigma_{2}}\right).
\]
By the volume conservation equation as expressed by the third of (\ref{eq:case-II})
we obtain
\[
\frac{kL}{2}\coth\frac{kL}{2}=1+\frac{(kL)^{2}}{(\sigma_{1}+\sigma_{2})L}
\]
By considering the function
\[
f(x)=\frac{x}{2}\frac{e^{x}+1}{e^{x}-1}-1-\frac{x^{2}}{c}
\]
where $x=kL$ and $c=(\sigma_{1}+\sigma_{2})L$, with Taylor series
expansion $f(x)=\left(\frac{1}{12}-\frac{1}{c}\right)x^{2}+\mathcal{O}(x^{4})$
about 0, and remarking that $f(x)>0$ in a neighborhood $]0,\delta[$
if $c>12$, while $f(x)$ becomes negative for $x$ sufficiently large,
we conclude the existence of a positive root of $f(x)=0$. When $L$
satisfies the condition (\ref{eq:crit-2}), from $\sigma_{1}\sigma_{2}\leqslant\frac{1}{4}(\sigma_{1}+\sigma_{2})^{2}$
we obtain $(\sigma_{1}+\sigma_{2})L>12$, and this completes the proof.
\end{proof}
We conclude by proving that sufficiently small partitions are stable.
\begin{prop}
\label{prop:no-nec-and-suf-cond}Let $\Omega$ and $M$ be as in Proposition
\ref{prop:crit-1}. If $L=|M|$ is sufficiently small, then $M$ is
stable.\end{prop}
\begin{proof}
If $L<L_{0}$, where $L_{0}=2\frac{\sigma_{1}+\sigma_{2}-\sqrt{\sigma_{1}^{2}+\sigma_{2}^{2}-\sigma_{1}\sigma_{2}}}{\sigma_{1}\sigma_{2}}$,
case  (III) has no solution. As in the proof of Proposition \ref{prop:crit-2}
we can prove that case (II) has also no solution, and thus there is
no eigenvalue in the range $\mu\leqslant-\kappa^{2}$. We are looking
for eigenvalues in the range $-\kappa^{2}<\mu\leqslant0$ or equivalently
$0<k\leqslant\kappa$. The determinant of the linear system (\ref{eq:case-I})
in the variables $B=-\frac{\lambda}{2k^{2}}$, $C$, $D$ is given
by
\[
\begin{alignedat}{1}D(x) & =\left|\begin{array}{ccc}
1 & \frac{k}{\sigma_{1}} & 1\\
1 & \sin x-\frac{k}{\sigma_{2}}\cos x & \cos x+\frac{k}{\sigma_{2}}\sin x\\
x & 1-\cos x & \sin x
\end{array}\right|\\
 & =2(1-\cos x)-\left(1+\frac{\sigma_{1}+\sigma_{2}}{k}\right)x\sin x+\frac{\sigma_{1}+\sigma_{2}}{k}x^{2}\cos x+\frac{\sigma_{1}\sigma_{2}}{k^{2}}x^{3}\sin x
\end{alignedat}
\]
with $x=kL$. We expand $\sin x$, $\cos x$ into a Taylor series
about $0$:
\[
D(x)=\frac{1}{\sigma_{1}\sigma_{2}}k^{2}x^{2}-\frac{1}{3}k\left(\frac{1}{\sigma_{1}}+\frac{1}{\sigma_{2}}\right)x^{3}+\frac{1}{12}\left(1-2\frac{k^{2}}{\sigma_{1}\sigma_{2}}\right)x^{4}+\mathcal{O}(x^{5}).
\]
Now assume $k$ is a root of $D(kL)=0$ in $]0,\kappa]$ for each
$L\leqslant L_{0}$, i.e. 
\begin{equation}
0=\frac{1}{\sigma_{1}\sigma_{2}}k^{2}-\frac{1}{3}k\left(\frac{1}{\sigma_{1}}+\frac{1}{\sigma_{2}}\right)x+\frac{1}{12}\left(1-2\frac{k^{2}}{\sigma_{1}\sigma_{2}}\right)x^{2}+\mathcal{O}(x^{3}).\label{eq:crit-3-eq-1}
\end{equation}
Since $k$ is bounded, in the limit $L\to0$ we have $x\to0$ and
for any convergent sequence $k_{n}\to k$ 
\[
0=\frac{1}{\sigma_{1}\sigma_{2}}k^{2}
\]
hence $k=0$. Thus in the limit $L\to0$ we have $x\to0$ and $k\to0$.
Substituting $x=kL$ in (\ref{eq:crit-3-eq-1}) gives
\begin{equation}
0=\frac{1}{\sigma_{1}\sigma_{2}}-\frac{1}{3}\left(\frac{1}{\sigma_{1}}+\frac{1}{\sigma_{2}}\right)L+\frac{1}{12}\left(1-2\frac{k^{2}}{\sigma_{1}\sigma_{2}}\right)L^{2}+\mathcal{O}(L^{3}).\label{eq:crit-3-eq-1-1}
\end{equation}
On taking the limit $L\to0$ we obtain
\[
0=\frac{1}{\sigma_{1}\sigma_{2}}
\]
which is absurd. Hence, possibly for a lower value for $L_{0}$, the
equation $D(kL)=0$ has no roots in $]0,\kappa]$ for all $L<L_{0}$. 
\end{proof}

\end{document}